\documentclass{article}
\usepackage{amsmath, amsfonts, amsthm, amssymb}
\usepackage{graphicx}
\usepackage{float}
\usepackage{overpic}
\usepackage{verbatim}% for comment
\usepackage{color}%for colouring comments
\usepackage[sort]{cite}

\numberwithin{equation}{section}

\newtheorem{thm}{Theorem}[section]
\newtheorem{lma}[thm]{Lemma}
\newtheorem{cor}[thm]{Corollary}

\renewcommand{\geq}{\geqslant}
\renewcommand{\leq}{\leqslant}

\newcommand{\Nbar}{\overline{N}}
\newcommand{\Cbar}{\overline{C}}
\newcommand{\Bbar}{\overline{B}}
\newcommand{\nbar}{\overline{n}}
\newcommand{\mbar}{\overline{m}}

\DeclareMathOperator{\R}{\mathbb{R}}

\DeclareMathOperator{\N}{\mathbb{N}}
\DeclareMathOperator{\E}{\mathbb{E}}

\allowdisplaybreaks

\title{The Assouad spectrum  of random self-affine carpets}

\author{Jonathan M.~Fraser\footnote{JMF was financially supported by the Leverhulme Trust Research
Fellowship RF-2016-500, the EPSRC Standard Grant EP/R015104/1, and the University of Waterloo.} 
\and Sascha Troscheit\footnote{ST was financially supported by NSERC Grants 2014-03154 and
2016-03719, and the University of Waterloo.}}

\begin{document}

\maketitle

\vspace{-5mm}
\begin{center}
	$^*$Mathematical Institute, University of St Andrews, North Haugh, St Andrews, Fife, KY16 9SS, UK.\\
	E-mail: jmf32@st-andrews.ac.uk
	\\ \vspace{3mm}
	$^\dagger$Department of Pure Mathematics, University of Waterloo, Waterloo, Ont., N2L 3G1,
	Canada.
	\\
	E-mail: stroscheit@uwaterloo.ca
	\end{center}

\begin{abstract}
We derive the almost sure Assouad spectrum and quasi-Assouad dimension of random self-affine
Bedford-McMullen carpets.  Previous work has revealed that the (related) Assouad dimension is not sufficiently
sensitive  to distinguish between subtle changes in the random model, since it tends to be almost
surely `as large as possible' (a deterministic quantity). This has been verified in conformal and
non-conformal settings.  In the conformal setting, the Assouad spectrum and quasi-Assouad
dimension  behave rather differently, tending to almost surely coincide with the upper box
dimension.  Here we investigate the non-conformal setting and find that the Assouad spectrum and
quasi-Assouad dimension generally do not coincide with the box dimension or Assouad dimension.  We
provide examples highlighting the subtle differences between these notions.  Our proofs combine
deterministic covering techniques with suitably adapted Chernoff estimates and Borel-Cantelli type
arguments. \\


\end{abstract}

\section{Assouad spectrum and quasi-Assouad dimension}

The Assouad dimension is an important notion of dimension designed to capture extreme local scaling
properties of a given metric space.  Its distance from the upper box dimension, which measures
average global scaling, can be interpreted as a quantifiable measure of inhomogeneity. Motivated by
this idea, Fraser and Yu introduced the Assouad spectrum which is designed to interpolate between
the upper box dimension and Assouad dimension, and thus reveal more precise geometric information
about the set, see \cite{Spectraa}.  Here we recall the basic definitions and, for concreteness, we
consider  non-empty compact  sets $F \subseteq \mathbb{R}^d$, although the general theory extends
beyond this setting.  For a bounded set $E \subseteq \mathbb{R}^d$ and a scale $r>0$ we let $N(E,r)$
be the minimum number of sets of diameter $r$ required to cover $E$.  The Assouad dimension of $F$
is defined by
\[
\dim_\text{A} F \ = \  \inf \Bigg\{ \alpha \  : \ ( \exists  \, C) \, (\forall \, 0<r<R<1) \, (\forall x \in F) \  N\big( B(x,R) \cap F, r \big) \ \leq \ C \bigg(\frac{R}{r}\bigg)^\alpha  \Bigg\}.
\]
The Assouad spectrum is the function defined by
\[
\theta  \ \mapsto \ \dim_{\mathrm{A}}^\theta F \ = \  \inf \bigg\{ \alpha \  : \   (\exists C>0) \, (\forall 0<R<1) \,  (\forall x \in F) \   N \big( B(x,R) \cap F ,R^{1/\theta} \big) \ \leq \ C \left(\frac{R}{R^{1/\theta}}\right)^\alpha \bigg\}
\]
where  $\theta$ varies over  $(0,1)$.  The related quasi-Assouad dimension is defined by
\[
\dim_{\mathrm{qA}} F \ = \  \lim_{\theta \to 1} \inf \bigg\{ \alpha \  : \   (\exists C>0) \, (\forall 0<r \leq R^{1/\theta} \leq R< 1) \,  (\forall x \in F) \  N \big( B(x,R) \cap F ,r \big) \ \leq \ C \left(\frac{R}{r}\right)^\alpha \bigg\}
\]
and the upper box dimension is defined by
\[
\overline{\dim}_\mathrm{B} F \ = \  \inf \Bigg\{ \alpha \  : \ ( \exists  \, C) \, (\forall \, 0<R<1) \ N\big(  F, r \big) \ \leq \ C \bigg(\frac{1}{r}\bigg)^\alpha  \Bigg\}.
\]
These dimensions are all related, but their relative differences can be subtle.  We summarise some important facts to close this section.  For any $\theta \in (0,1)$ we have
\[
\overline{\dim}_\text{B} F \leq  \dim_{\mathrm{A}}^\theta F \leq \dim_{\mathrm{qA}} F  \leq \dim_{\mathrm{A}} F  
\]
and any of these inequalities can be strict.  Moreover, the Assouad spectrum is a continuous function of $\theta$ and also satisfies
\begin{equation} \label{spectrumbound}
 \dim_{\mathrm{A}}^\theta F  \leq \frac{\overline{\dim}_\text{B} F}{1-\theta}.
\end{equation}
We also note that for a given $\theta$ it is not necessarily true that the Assouad spectrum is given
by the expression after the limit in the definition of the quasi-Assouad dimension: this notion is
by definition monotonic in $\theta$ but the spectrum is not necessarily monotonic \cite[Section
8]{Spectraa}. However, it has recently been shown in \cite{ScottishCanadian} that $\dim_{\mathrm{qA}} F  = \lim_{\theta \to
1} \dim_{\mathrm{A}}^\theta F$ and, combining this with (\ref{spectrumbound}), we see that the
Assouad spectrum necessarily interpolates between the upper box dimension and the quasi-Assouad
dimension.

For more information, including basic properties, concerning  the upper box dimension, see
~\cite[Chapters 2-3]{Falconer}.  For the Assouad dimension, see \cite{Fraser14, Luukkainen,
Robinson}, the quasi-Assouad dimension, see \cite{Hare}, and for the Assouad spectrum, see
\cite{Spectraa, Spectrab, ScottishCanadian}.

\section{Self-affine carpets: random and deterministic}

In this paper we consider random self-affine carpets. More specifically, random 1-variable analogues
of the self-affine sets introduced by Bedford and McMullen in the 1980s.  In the deterministic
setting, the box dimensions were computed independently  by Bedford and McMullen \cite{Bedford,
McMullen} and the Assouad dimension was computed by Mackay \cite{Mackay}.  The Assouad spectrum was
computed by Fraser and Yu \cite{Spectrab}, and these results also demonstrated that the
quasi-Assouad and Assouad dimensions coincide by virtue of the spectrum reaching the Assouad
dimension.  In the random setting, the (almost sure) box dimensions were first computed by Gui and
Li \cite{GuiLi} for fixed subdivisions and by Troscheit~\cite{Troscheit18} in the most general
setting that we are aware of.
The (almost sure) Assouad dimension was computed by Fraser, Miao and Troscheit
\cite{FraserMiaoTroscheit}.  In this article we compute the quasi-Assouad dimension and the Assouad
spectrum in the random setting.  Unlike in the deterministic case, we find that the quasi-Assouad
dimension and Assouad dimension are usually almost surely distinct.
Further, the quasi-Assouad dimension is in general also distinct from the box  dimension.  This is in stark contrast to the conformal setting, where it was shown that the quasi-Assouad dimension (and thus Assouad spectrum) is almost surely equal to the upper box dimension (and distinct from the Assouad dimension), see \cite{Troscheit17}.

We close this section by describing our model. Let $\Lambda = \{ 1, \dots, |\Lambda |\}$ be a finite
index set and for each $i \in \Lambda$ fix integers $n_i >  m_i \geq 2$ and divide the unit square
$[0,1]^2$ into a uniform $m_i \times n_i$ grid.  For each $i \in \Lambda$ let $\mathcal{I}_i$ be a
non-empty subset of the set of $m_i^{-1} \times n_i^{-1}$ rectangles in the grid and let $N_i =
|\mathcal{I}_i|$.  Let $B_i$ be the number of distinct columns which contain rectangles from
$\mathcal{I}_i$, and $C_i$ be the maximum number of rectangles in $\mathcal{I}_i$ which are contained
in a particular column.  Note that $1 \leq B_i \leq m_i$,  $1 \leq C_i \leq n_i$ and $N_i \leq B_i
C_i$.  For each rectangle $j \in \mathcal{I}_i$, let $S_j$ be the unique orientation preserving
affine map which maps the unit square $[0,1]^2$ onto the rectangle $j$.

Let $\Omega = \Lambda^\mathbb{N}$ and for each  $\omega = (\omega_1, \omega_2, \dots) \in \Omega$, we are interested in the corresponding attractor
\[
F_\omega \, = \, \bigcap_{k \geq 1}\  \bigcup_{j_1\in \mathcal{I}_{\omega_1}, \dots, j_k \in\mathcal{I}_{\omega_k}} S_{j_1} \circ \dots \circ S_{j_k}\left( [0,1]^2 \right).
\]
By randomly choosing $\omega \in \Omega$, we randomly choose an attractor $F_\omega$ and we wish to make statements about the generic nature of $F_\omega$. For this, we need a measure on $\Omega$.  Let $\{p_i\}_{i \in \Lambda}$ be a set of probability weights, that is, for each $i \in \Lambda$, $0<p_i<1$ and $\sum_{i \in \Lambda} p_i = 1$.  We extend these basic probabilities to a Borel measure $\mathbb{P}$ on $\Omega$ in the natural way, which can be expressed as the infinite product measure
\[
\mathbb{P} = \prod_{k \in \mathbb{N}} \sum_{i \in \Lambda} p_i \delta_{i}\,\,,
\]
where $\Omega$ is endowed with the product topology and $\delta_i$ is a unit mass concentrated at $i \in \Lambda$.  Note that the deterministic model can be recovered if $|\Lambda | = 1$, that is, there is only one ``pattern" available, which is therefore chosen at every stage in the process. In this case, the  deterministic attractor is the unique non-empty set $F \subseteq [0,1]^2$ satisfying
\[
F = \bigcup_{j \in \mathcal{I}_1} S_j(F).
\]

\section{Results}

Our main result is an explicit formula which gives the Assouad spectrum of our random self-affine sets almost surely.  For simplicity we suppress summation over $i \in \Lambda$ to simple summation over $i$ throughout this section.

\begin{thm} \label{CarpetsSpectra}
For $\mathbb{P}$ almost  all $\omega \in \Omega$, we have
\[
  \dim_{\mathrm{A}}^\theta F_\omega \ = \  \left\{ \begin{array}{ccc}
\dfrac{1}{1- \theta}
  \left( \dfrac{\sum_{i  } p_i \log \left( B_iC_i^\theta N_i^{-\theta} \right) }{ \sum_{i  } p_i\log m_i}  +   \dfrac{\sum_{i  } p_i \log \left(N_iB^{-1}_iC_i^{-\theta}\right)}{\sum_{i  } p_i \log n_i} \right)  & \colon & 0< \theta \leq \dfrac{\sum_{i  } p_i  \log m_i}{\sum_{i  } p_i \log n_i} \\ \\
 \dfrac{\sum_{i  } p_i\log B_i }{\sum_{i  } p_i\log m_i}  +   \dfrac{\sum_{i  } p_i\log C_i}{\sum_{i
 } p_i\log n_i} &   \colon &  \dfrac{\sum_{i  } p_i  \log m_i}{\sum_{i  } p_i \log n_i}< \theta
 <1\,,
\end{array}  \right.
\]
where $F_\omega$ is the $1$-variable random Bedford-McMullen carpet associated with $\omega \in \Omega$.
\end{thm}

As an immediate consequence of Theorem \ref{CarpetsSpectra} we obtain a formula for the quasi-Assouad dimension which holds almost surely. 

\begin{cor} \label{CarpetsQuasi}
For $\mathbb{P}$ almost  all $\omega \in \Omega$, we have
\[
\dim_{\mathrm{qA}}F_\omega \ = \   \frac{\sum_{i } p_i\log B_i }{\sum_{i } p_i\log m_i}  +
\frac{\sum_{i } p_i\log C_i}{\sum_{i } p_i\log n_i}\,,
\]
where $F_\omega$ is the $1$-variable random Bedford-McMullen carpet associated with $\omega \in \Omega$.
\end{cor}

\begin{proof}
This follows immediately from Theorem \ref{CarpetsSpectra} and the fact that 
$\dim_{\mathrm{A}}^\theta E \to \dim_{\mathrm{qA}} E $ as $\theta \to 1$ for any set
$E\subseteq{\R^d}$, see
\cite[Corollary 2.2]{ScottishCanadian}.
\end{proof}

Note that the result in \cite{FraserMiaoTroscheit} states that for $\mathbb{P}$ almost  all $\omega \in \Omega$, we have
\[
\dim_\textup{A} F_\omega \  =  \  \max_{i\in\Lambda} \, \frac{\log B_i}{\log m_i} \, + \, \max_{i\in\Lambda} \,  \frac{\log C_i}{\log n_i}.
\]
Therefore Corollary \ref{CarpetsQuasi} demonstrates the striking difference between the Assouad and
quasi-Assouad dimensions in the random setting. In particular the almost sure value of the Assouad
dimension does not depend on the weights $\{p_i\}_{i \in \Lambda}$, but the almost sure value of the
quasi-Assouad dimension depends heavily on the weights.  The almost sure value of the Assouad
dimension is also extremal in the sense that it is the maximum over all realisations $\omega \in
\Omega$, whereas the quasi-Assouad dimension is an average.  Recall that the quasi-Assouad and
Assouad dimensions always  coincide for deterministic self-affine carpets, see  \cite{Spectrab}.

It is worth noting that the Assouad dimension of the random attractor is at least the maximal
Assouad dimension of the deterministic attractors, whereas the quasi-Assouad dimension is bounded
above by the dimension of the individual attractors. That is, letting $\underline{i}=(i,i,i,\dots)$, 
\[
  \dim_{\mathrm{A}} F_\omega \geq \max_{i\in\Lambda}\dim_{\mathrm{A}} F_{\underline
  i}\;\;(a.s.)\quad\text{and}\quad
  \dim_{\mathrm{qA}} F_\omega \leq \max_{i\in\Lambda}\dim_{\mathrm{qA}} F_{\underline i}\;\;(a.s.).
\]
Typically these inequalities are strict and it is further possible that 
$\dim_{\mathrm{qA}} F_\omega < \min_{i\in\Lambda}\dim_{\mathrm{qA}} F_{\underline i}$\,, almost
surely, see Figure~\ref{fig:spectrum} and the example in Section~\ref{sect:extremeExample}.

Finally, note that the almost sure values of the Assouad and quasi-Assouad dimensions coincide if
and only if there exists $\alpha, \beta \in (0,1]$ such that for all $i \in \Lambda$ we have
$\frac{\log B_i}{\log m_i} = \alpha$ and $\frac{\log C_i}{\log n_i} = \beta$.  This follows by
considering   `weighted mediants'.  In particular, the two terms giving the quasi-Assouad dimension
are weighted mediants of the fractions $\frac{\log B_i}{\log m_i}$ and $\frac{\log C_i}{\log n_i}$
respectively.  It is well-known that weighted mediants are equal to the maximum if and only if all
the fractions coincide.  In particular, coincidence of all of the deterministic Assouad (and quasi-Assouad) 
dimensions is \emph{not} sufficient to ensure almost sure coincidence of the Assouad and
quasi-Assouad dimensions in the random case.

Simple algebraic manipulation yields the following random analogue of \cite[Corollary
3.5]{Spectrab}.  In particular, the random variable $\dim_{\mathrm{A}}^\theta F_\omega$ can be
expressed in terms of the random variables $ \overline{\dim}_\mathrm{B} F_\omega $ and  $\dim_\mathrm{qA} F_\omega $.

\begin{cor}
 For $\mathbb{P}$ almost  all $\omega \in \Omega$, we have
\[
  \dim_{\mathrm{A}}^\theta F_\omega \ = \   \min \left\{ \frac{   \overline{\dim}_\mathrm{B}
  F_\omega \ -  \ \theta \, \left(\dim_\mathrm{qA} F_\omega  - \left( \dim_\mathrm{qA} F_\omega -
  \overline{\dim}_\mathrm{B} F_\omega\right) \frac{\sum_{i  } p_i  \log n_i}{\sum_{i  } p_i \log m_i} \right) }{1- \theta} , \ \dim_\mathrm{qA} F \right\}
\]
where $F_\omega$ is the $1$-variable random Bedford-McMullen carpet associated with $\omega \in \Omega$.
\end{cor}

Note that \cite[Corollary 3.5]{Spectrab} is formulated using the Assouad dimension instead of the
quasi-Assouad dimension (although they are equal in the deterministic case).  Our result shows that
the quasi-Assouad dimension is really  the `correct' notion to use here.

\subsection{Generic Example}
\label{sect:genericExample}
For illustrative purposes, we exhibit a representative example and provide pictures of the random and deterministic carpets along with their spectra.
Let $\Lambda=\{1,2\}$ and $\mathbb{P}$ be the $1/2$--$1/2$ Bernoulli probability measure on
$\Omega=\Lambda^{\N}$. That is, we consider two iterated function systems that we choose with equal
probability. 

The first iterated function system consists of $N_1=20$ maps, where the unit square is divided into
$m_1=19$ by $n_1=21$ rectangles. There are $B_1=10$ columns containing at least one rectangle and the
maximal number  of rectangles in a particular column is $C_1=8$. For the attractor of this deterministic
Bedford-McMullen carpet we obtain:
\[
  \dim_{\mathrm{A}}F_{\underline{1}}=\dim_{\mathrm{qA}}F_{\underline{1}}=\frac{\log10}{\log19}+\frac{\log8}{\log21}\approx1.465
  \quad\text{and}\quad
  \overline{\dim}_{\mathrm{B}}F_{\underline{1}}=\frac{\log10}{\log19}+\frac{\log2}{\log21}\approx1.010
\]
and the spectrum interpolates between these two values with a phase transition at $\log19 / \log21 \approx 0.967$.
The spectrum is plotted in Figure~\ref{fig:spectrum} and the attractor is shown in
Figure~\ref{FigureDet1}.

\begin{figure}[tb]
  \begin{center}
    \begin{overpic}[width=35em]{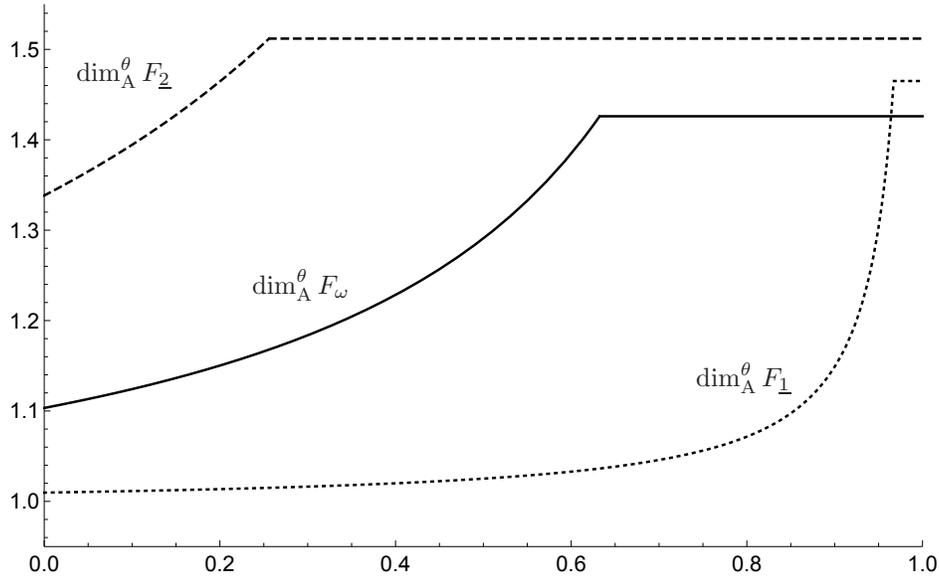}
      \put(26,30){$\dim_{\mathrm{A}}^\theta F_\omega$}
      \put(7,53){$\dim_{\mathrm{A}}^\theta F_{\underline{2}}$}
      \put(74,20){$\dim_{\mathrm{A}}^\theta F_{\underline{1}}$}
    \end{overpic}
\end{center}
\caption{The Assouad spectra of the sets in the example of Section~\ref{sect:genericExample}.  The deterministic spectra are shown in dashed lines and the almost sure spectrum in the random case is given by a solid line.}
  \label{fig:spectrum}
\end{figure}

{\setlength{\fboxsep}{0pt}\setlength{\fboxrule}{0.5pt}
\begin{figure}[tb]
  \begin{center}
    \fbox{\includegraphics[width=14em]{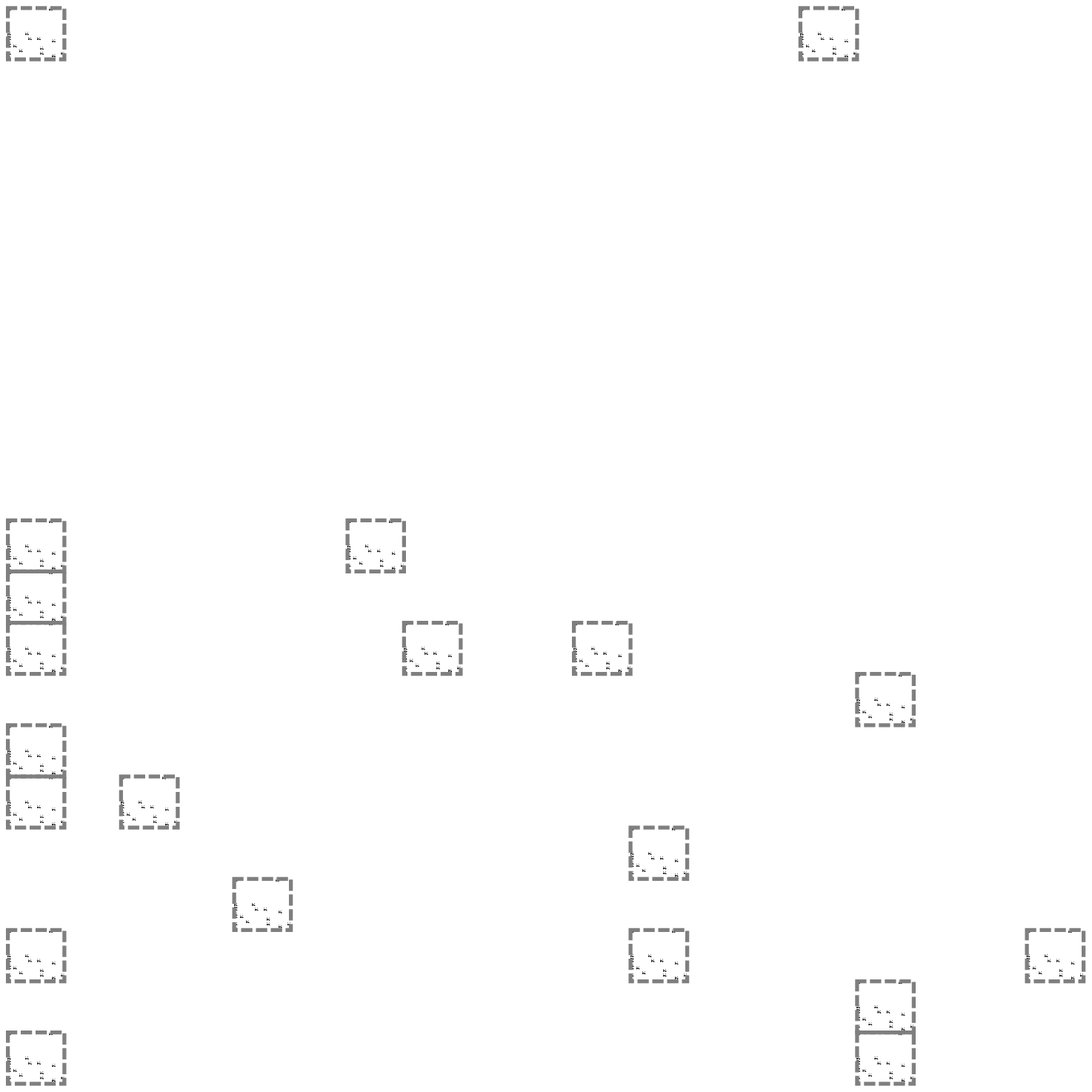}}
    \fbox{\includegraphics[width=14em]{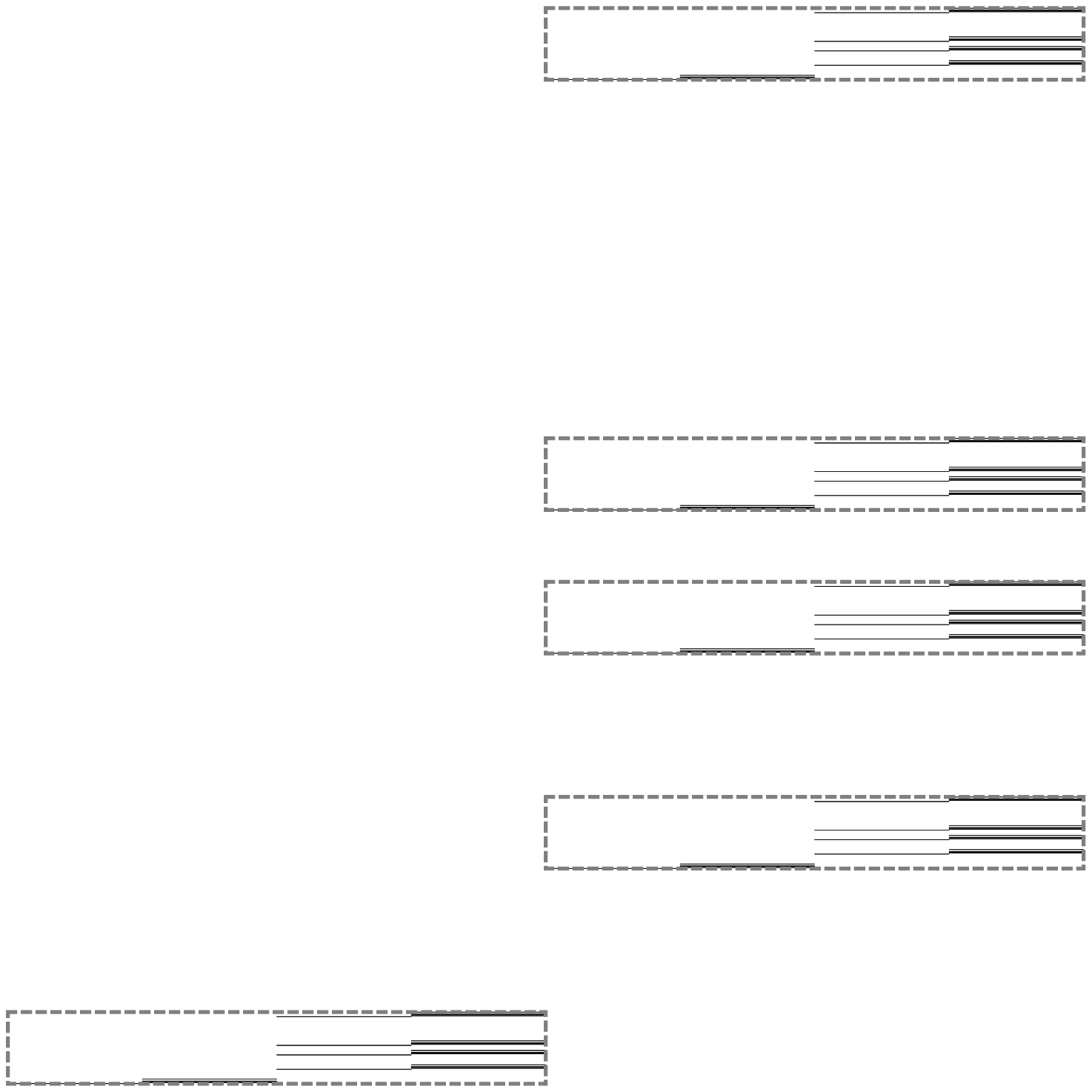}}
    \fbox{\includegraphics[width=14em]{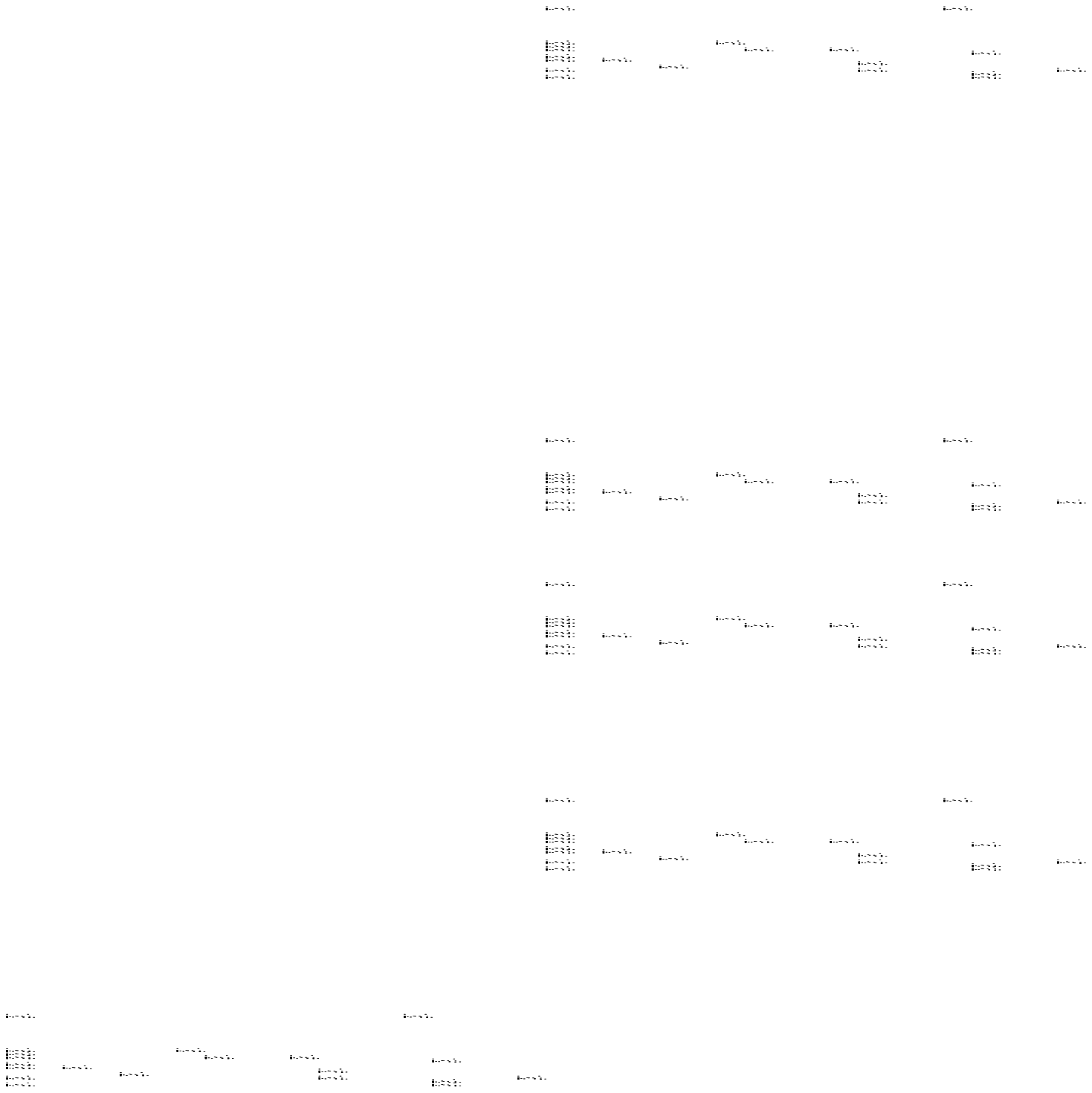}}
\end{center}
\caption{The attractors $F_{\underline{1}}$, $F_{\underline{2}}$, and $F_{\omega}$ for
$\omega=(2,1,1,2,1,\dots)$ as used in the example in Section~\ref{sect:genericExample}.}
  \label{FigureDet1}
\end{figure}
}

The second iterated function system consists of $N_2=5$ maps, where the unit square is divided into
$m_2=2$ by $n_2=15$ rectangles. There are $B_2=2$ columns containing at least one rectangle  and the maximal number of rectangles in a particular column is $C_2=4$. For the attractor of this deterministic Bedford-McMullen
carpet we obtain:
\[
  \dim_{\mathrm{A}}F_{\underline{2}}=\dim_{\mathrm{qA}}F_{\underline{2}}=1+\frac{\log4}{\log15}\approx1.512
  \quad\text{and}\quad
  \overline{\dim}_{\mathrm{B}}F_{\underline{2}}=1+\frac{\log5/2}{\log15}\approx1.338
\]
and the spectrum interpolates between these two values with a phase transition at $\log2 / \log15 \approx 0.256$.
The spectrum is plotted in Figure~\ref{fig:spectrum} and the attractor is shown in
Figure~\ref{FigureDet1}.

Our results now give the following values for almost every $\omega\in\Omega$:
\[
  \dim_{\mathrm{A}}F_\omega=1+\frac{\log8}{\log21}\approx 1.683
  ,\quad
  \dim_{\mathrm{qA}}F_\omega=\frac{\log20}{\log38}+\frac{\log32}{\log315}\approx1.426,
\]
\[
  \text{and}\quad
  \overline{\dim}_{\mathrm{B}}F_\omega=\frac{\log20}{\log38}+\frac{\log5}{\log315}\approx1.103.
\]
We note that in this example the almost sure value of the Assouad dimension exceeds that of the
individual attractors, the almost sure quasi-Assouad dimension is less than the quasi-Assouad
dimensions of the individual attractors, and that the phase transition in the spectrum occurs at
$\log38/\log315\approx0.632$.

\subsection{An extreme example}
\label{sect:extremeExample}

By constructing explicit examples, we demonstrate the following interesting phenomenon, which
highlights the subtle difference between the quasi-Assouad and Assouad dimensions.  For all
$\varepsilon \in (0,1)$, there exist two deterministic self-affine carpets $E,F$ with
$\overline{\dim}_\textup{B} E = \overline{\dim}_\textup{B} F = \dim_\textup{qA} E = \dim_\textup{qA}
F = \dim_\textup{A} E = \dim_\textup{A} F = 1$ such that when one mixes the two constructions by
randomising as above, one finds that almost surely
\[
\dim_\textup{qA} F_\omega \leq \varepsilon< 2 = \dim_\textup{A} F_\omega.
\]
Let $\varepsilon \in (0,1)$, $\Lambda = \{1,2\}$, $m_1=2$, $n_1 = n$, $m_2=m$, $n_2 = m+1$, where
$m,n$ are large integers which will be chosen later depending only on $\varepsilon$. Let
$\mathcal{I}_1$ consist of both rectangles from a particular row in the first grid and
$\mathcal{I}_2$ consist of all $(m+1)$  rectangles in a particular column of the second grid.  The
deterministic carpets associated with these systems are both unit line segments: a horizontal line
in the first case and a vertical line in the second.  Therefore both have all the dimensions we
consider being equal to 1. Let $p_1=p_2=1/2$, although the precise choice of weights is not
particularly important. It follows that for $\mathbb{P}$ almost all $\omega \in \Omega$ we have
\[
\dim_{\mathrm{qA}}F_\omega \ = \  \frac{(1/2) \log 2 + (1/2) \log 1 }{(1/2)\log 2+ (1/2) \log m}  +
\frac{(1/2) \log 1 + (1/2) \log ( m+1)}{(1/2)\log n+ (1/2) \log (m+1)} = \frac{\log 2 }{\log (2 m)}
+   \frac{\log (m+1)}{\log n(m+1)}.
\]
Choose $m$ sufficiently large to ensure that $ \frac{\log 2 }{\log (2 m)} \leq \varepsilon/2$ and,
now that $m$ is fixed, choose $n$ sufficiently large to ensure that $\frac{\log (m+1)}{\log n(m+1)}
\leq \varepsilon/2$.  The main result in  \cite[Theorem 3.2]{FraserMiaoTroscheit} gives that for any choice of
$m,n \geq 2$, $\dim_{\mathrm{A}}F_\omega = 2$ almost surely, and therefore the desired result
follows.

\section{Proofs}

\subsection{Approximate squares}

In this section we introduce (random) approximate squares, which are a common object in the study of
self-affine carpets. Fix  $\omega = (\omega_1, \omega_2, \dots) \in \Omega$,  $R \in (0,1)$ and let
$k_1^\omega(R)$ and $k_2^\omega(R)$ be the unique positive integers satisfying
\begin{equation} \label{k1def}
\prod_{l=1}^{k_1^\omega(R)} m_{\omega_l}^{-1} \, \leq \,  R \, <  \, \prod_{l=1}^{k_1^\omega(R)-1} m_{\omega_l}^{-1}
\end{equation}
and
\begin{equation} \label{k2def}
\prod_{l=1}^{k_2^\omega(R)} n_{\omega_l}^{-1} \, \leq \,  R \, <  \, \prod_{l=1}^{k_2^\omega(R)-1}
n_{\omega_l}^{-1},
\end{equation}
respectively. Also let
\[
m_{\max} = \max_{i \in \Lambda} \,  m_i \qquad \text{and} \qquad n_{\max} = \max_{i \in \Lambda} \,  n_i.
\]
A rectangle $[a,b] \times [c,d] \subseteq [0,1]^2$ is called an \emph{approximate $R$-square} if it is of the form
\[
S \big( [0,1]^2 \big) \cap \Big( \pi_1 \big( T\big( [0,1]^2 \big) \big) \times [0,1] \Big),
\]
where $\pi_1 : (x,y) \mapsto x$ is the projection onto the first coordinate and
\[
S \ = \ S_{i_1} \circ \cdots \circ S_{ i_{k_2^\omega(R)}}
\]
and
\[
T \ = \ S_{ i_1} \circ \cdots \circ S_{ i_{k_1^\omega(R)}},
\]
for some common sequence $i_1, i_2, \dots $ with $i_j \in \mathcal{I}_{\omega_j}$ for all $j$. Here we say $Q$ is \emph{associated with the sequence} $i_1, i_2, \dots $, noting that the entries $i_1, i_2, \dots ,  i_{k_1^\omega(R)}$ determine $Q$.  In particular, the base
\[
b-a  \ =  \ \prod_{i=1}^{k_1^\omega(R)} m_{\omega_i}^{-1} \ \in \ (m_{\max}^{-1}R , R] \qquad \qquad \text{by (\ref{k1def})}
\]
and the height
\[
d-c  \ =  \ \prod_{i=1}^{k_2^\omega(R)} n_{\omega_i}^{-1} \ \in \ (n_{\max}^{-1}R , R] \qquad \qquad
\text{by (\ref{k2def})}\,,
\]
and so approximate $R$-squares are indeed approximately  squares with base and height uniformly comparable to $R$, and therefore each other.

\subsection{Proof strategy and notation}
\label{sect:proofStrategy}

In order to simplify the exposition of  our proofs, we define the following weighted geometric
averages of the important parameters:
\[
\Nbar = \prod_{i \in \Lambda} N_i^{p_i} \qquad \Bbar = \prod_{i \in \Lambda} B_i^{p_i} \qquad \Cbar = \prod_{i \in \Lambda} C_i^{p_i} \qquad \mbar = \prod_{i \in \Lambda} m_i^{p_i}\qquad \nbar = \prod_{i \in \Lambda} n_i^{p_i}.
\]
Using this notation, in order to prove our result it is sufficient to prove the following two statements:
\begin{enumerate}
\item[(1)] For all $\log \mbar/\log \nbar< \theta <1  $ we have that for $\mathbb{P}$ almost all $\omega \in \Omega$
\[
\dim_{\mathrm{A}}^\theta F_\omega \ \leq  \    \frac{\log \Bbar }{\log \mbar}  +   \frac{\log \Cbar}{\log \nbar}.
\]
\item[(2)] For all $0<\theta < \log \mbar/\log \nbar $  we have that for $\mathbb{P}$ almost all $\omega \in \Omega$
\[
\dim_{\mathrm{A}}^\theta F_\omega \ = \   \frac{ 1}{1- \theta} \left( \frac{\log \Bbar }{\log \mbar}
+   \frac{\log \Nbar/\Bbar}{\log \nbar} \right)  \ -  \ \frac{ \theta}{1- \theta} \left( \frac{\log
(\Nbar/\Cbar)}{\log \mbar} \, + \,  \frac{\log \Cbar }{\log \nbar}\right)  .
\]
\end{enumerate}
To see why this is sufficient, first note that since the Assouad spectrum is a continuous function
in $\theta$, see \cite[Corollary 3.5]{Spectraa}, it is determined by its values on a countable dense
set and so the above statements imply the \emph{a priori} stronger statements that for $\mathbb{P}$
almost all $\omega \in \Omega$, we have the given estimates for \emph{all} $\theta$. Secondly,
since the Assouad spectrum necessarily approaches the quasi-Assouad dimension as $\theta \to 1$, (1)
demonstrates that the quasi-Assouad dimension is at most
\[
 \frac{\log \Bbar }{\log \mbar}  +   \frac{\log \Cbar}{\log \nbar}.
\]
and since (2) demonstrates that the Assouad spectrum attains this value at $ \theta = \log
\mbar/\log \nbar $, it follows from \cite[Corollary 3.6]{Spectraa} that it is constant in the
interval $[ \log \mbar/\log \nbar ,1)$.   Technically speaking \cite[Corollary 3.6]{Spectraa} proves
that if the Assouad spectrum is equal to  the \emph{Assouad} dimension at some $\theta' \in (0,1)$,
then it is constant in the interval $[\theta',1)$, but the same proof allows one to replace the
  Assouad dimension with the quasi-Assouad dimension in this statement.

Finally, note that to establish estimates for $\dim_{\mathrm{A}}^\theta F_\omega$, it suffices to
replace balls of radius $R$ with approximate $R$-squares in the definition.  That is, to estimate
$N\left(Q \cap F_\omega, R^{1/\theta} \right) $ where $Q$ is associated to $i_1, i_2, \dots $ with
$i_j \in \mathcal{I}_{\omega_j}$ for all $j$ instead of $N\left(B(x,R) \cap F_\omega, R^{1/\theta}
\right) $ for $x \in F_\omega$.  This is because balls and approximate squares are comparable and
one can pass  covering estimates concerning one to covering estimates concerning the other up to
constant factors.  This duality is standard and we do not go into the details.

\subsection{Covering estimates}

Let $\omega \in \Omega$, $\theta \in (0,1)$,  $R \in (0,1)$ and $Q$ be an approximate $R$-square
associated with the sequence $i_1, i_2, \dots $ with $i_j \in \mathcal{I}_{\omega_j}$ for all $j$.
In what follows we describe sets of the form $S_{j_1} \circ \cdots \circ S_{ j_{l}}(F_\omega)$ as
level $l$ \emph{cylinders} and level $(l+1)$ cylinders lying inside a particular level $l$ cylinder
will be referred to as \emph{children}.  Moreover, \emph{iteration} will refer to moving attention
from a particular cylinder, or collection of cylinders, to the cylinders at the next level. 

We wish to estimate $N\left(Q \cap F_\omega, R^{1/\theta} \right) $ and to do this we decompose
$Q \cap F_\omega$ into cylinders at level $k^\omega_2\left(R^{1/\theta}\right)$ and cover each
cylinder independently. Therefore we first need to count how many level
$k^\omega_2\left(R^{1/\theta}\right)$ cylinders lie inside $Q$.  There are two cases, which we
describe separately.  
\\ \\
  \emph{Case (i)}: $k^\omega_1(R) < k^\omega_2\left(R^{1/\theta}\right)$.\\
We start by noting that $Q$ lies inside a (unique) level $k^\omega_2(R) $ cylinder.  As we move to
the next level only the children of this cylinder lying in a particular `column' will also intersect
$Q$. Iterating inside cylinders intersecting $Q$ until level $k^\omega_1(R)$ yields a  decomposition
of $Q$ into several cylinders arranged in a single column each of which has base the same length as
that of $Q$.  The number of these cylinders is at most
\[
\prod_{l=k_2^\omega(R)+1}^{k^\omega_1(R)} C_{\omega_l}\,,
\]
since each iteration from the $(l-1)$th level to the $l$th multiplies the number of cylinders
intersecting $Q$ at the previous level by the number of rectangles in a particular column of
$\mathcal{I}_{\omega_l}$ system, which is, in particular, bounded above by $C_{\omega_l}$.  The
situation is simpler from this point on.  We continue to iterate inside each of the level
$k^\omega_1(R)$ cylinders until level $k^\omega_2\left(R^{1/\theta}\right)$, but this time all of
the children remain inside $Q$ at every iteration.  Therefore we find precisely
\[
\prod_{l=k^\omega_1(R)+1}^{k^\omega_2\left(R^{1/\theta}\right)}  N_{\omega_l} 
\]
level $k^\omega_2\left(R^{1/\theta}\right)$ cylinders inside each level $k^\omega_1(R)$ cylinder.
As mentioned above, we now cover each of these cylinders individually.  To do this, we further
iterate inside each such cylinder until level $k_1^\omega\left(R^{1/\theta}\right)$ and group
together cylinders at this level which lie in the same column.  This decomposes the level
$k_1^\omega\left(R^{1/\theta}\right)$ cylinders into approximate $R^{1/\theta}$ squares, each of
which can be covered by 4 balls of diameter $R^{1/\theta}$.  Therefore it only remains to count the
number of distinct level $k_1^\omega\left(R^{1/\theta}\right)$ columns inside a level
$k^\omega_2\left(R^{1/\theta}\right)$ cylinder.  Iterating from the $(l-1)$th level to the $l$th
level  multiplies the number of columns by $B_{\omega_l}$ and therefore the number is 
\[
\prod_{l=k^\omega_2\left(R^{1/\theta}\right)+1}^{k_1^\omega\left(R^{1/\theta}\right)}  B_{\omega_l} .
\]
Combining the three counting arguments from above yields
\begin{equation}\label{eq:upperCounting1}
N\left(Q \cap F_\omega, R^{1/\theta} \right)  \ \leq \ 4
 \left(\prod_{l=k_2^\omega(R)+1}^{k^\omega_1(R)} C_{\omega_l} \right) 
\ \left(\prod_{l=k^\omega_1(R)+1}^{k^\omega_2\left(R^{1/\theta}\right)}  N_{\omega_l} \right) 
\ \left( \prod_{l=k^\omega_2\left(R^{1/\theta}\right)+1}^{k_1^\omega\left(R^{1/\theta}\right)}  B_{\omega_l} \right) .
\end{equation}
Moreover, this estimate is sharp in the sense that we can always find a particular  approximate $R$-square $Q$ such that
\begin{equation}\label{eq:lowerCounting1}
N\left(Q \cap F_\omega, R^{1/\theta} \right)  \ \geq \ K
 \left(\prod_{l=k_2^\omega(R)+1}^{k^\omega_1(R)} C_{\omega_l} \right) 
\ \left(\prod_{l=k^\omega_1(R)+1}^{k^\omega_2\left(R^{1/\theta}\right)}  N_{\omega_l} \right) 
\ \left( \prod_{l=k^\omega_2\left(R^{1/\theta}\right)+1}^{k_1^\omega\left(R^{1/\theta}\right)}
B_{\omega_l} \right)\,,
\end{equation}
for some constant $K>0$ depending on $m_{\max}$ and $n_{\max}$.  Such a $Q$ is provided by any
approximate $R$-square where $T = S_{ i_1} \circ \cdots \circ S_{ i_{k_1^\omega(R)}}$ is chosen such
that each map $i_j$ lies in a maximal column of $\mathcal{I}_j$, that is a column consisting of
$C_j$ rectangles.  Finally, the small constant $K$ in the lower bound appears since a single ball of
diameter $R^{1/\theta}$ can only intersect at most a constant number of the approximate
$R^{1/\theta}$ squares found above and therefore counting approximate $R^{1/\theta}$ squares is
still comparable to counting optimal  $R^{1/\theta}$ covers.
\\ \\
  \emph{Case (ii)}: $k^\omega_1(R) \geq k^\omega_2\left(R^{1/\theta}\right)$.\\
The distinctive feature of this case is that when one iterates inside the  level $k^\omega_2(R) $
cylinder containing $Q$, one reaches the situation where the height of the cylinders is roughly
$R^{1/\theta}$ (level $k^\omega_2\left(R^{1/\theta}\right) $) \emph{before} the cylinders lie
completely inside $Q$ (level $k^\omega_1(R)$).  This means that the middle term in the above product
no longer appears.  The rest of the argument is similar, however, and we end up with
\begin{equation}\label{eq:upperCounting2}
N\left(Q \cap F_\omega, R^{1/\theta} \right)  \ \leq \ 4
 \left(\prod_{l=k_2^\omega(R)+1}^{k^\omega_2\left(R^{1/\theta}\right)} C_{\omega_l} \right) 
\ \left( \prod_{l=k^\omega_1\left(R\right)+1}^{k_1^\omega\left(R^{1/\theta}\right)}  B_{\omega_l} \right) .
\end{equation}
One subtle feature of this estimate is that we appear to skip from level
$k^\omega_2\left(R^{1/\theta}\right)$ to level $k^\omega_1\left(R\right)$.  This is to avoid
over-counting due to the fact that, inside a level $k^\omega_2\left(R^{1/\theta}\right)$  cylinder
intersecting $Q$, only a  single level $k^\omega_1\left(R\right)$ column actually lies inside $Q$,
and can thus contribute to the covering number.  This column comprises of several
$k^\omega_1\left(R\right)$ cylinders and, since the height of this column is comparable to
$R^{1/\theta}$, to cover this column efficiently one only needs to count the number of level
$k_1^\omega\left(R^{1/\theta}\right)$ columns inside a single level $k^\omega_1\left(R\right)$
cylinder.  This gives the second multiplicative term in the estimate, which concerns the terms
$B_{\omega_l}$.

Once again, this bound is sharp in the sense that there exists an approximate $R$-square $Q$ such that
\[
N\left(Q \cap F_\omega, R^{1/\theta} \right)  \ \geq \ K
 \left(\prod_{l=k_2^\omega(R)+1}^{k^\omega_2\left(R^{1/\theta}\right)} C_{\omega_l} \right) 
\ \left( \prod_{l=k^\omega_1\left(R\right)+1}^{k_1^\omega\left(R^{1/\theta}\right)}  B_{\omega_l} \right).
\]

\subsection{Proof of the Main Theorem}
We start our proof with this lemma, which is a simple variant of a Chernoff bound for stopped sums of
random variables.  We write $\mathbb{P}\{a \geq b\}$ to denote $\mathbb{P}\left( \left\{ \omega \in \Omega : a \geq b\right\} \right)$ and write $\mathbb{E}(\cdot)$ for the expectation of a random variable with respect to $\mathbb{P}$.
\begin{lma}\label{lma:randomUpperLemma}
  Let $\{X_i\}$ be a sequence of non-negative discrete i.i.d.\ random variables with finite expectation
  $0<\overline{X}=\E(X)<\infty$.
  Let $\widehat{k}\in\N$ and let $k\leq \widehat{k}$ be a random variable. Let $\tau>\widehat{k}$ be a stopping time
  with finite expectation. 
  Then, for all $\varepsilon,t>0$,
  \begin{equation}\label{eq:firstUpperEstimate}
    \mathbb{P}\left\{ \sum_{i=k}^{\tau}X_i \geq (1+\varepsilon)(\tau-k+1)\overline{X} \right\}
    \leq \E\left( \E\left( e^{t(X-(1+\varepsilon)\overline{X})} \right)^{\tau-k} \right)
  \end{equation}
  and
  \begin{equation}  
    \mathbb{P}\left\{ \sum_{i=k}^{\tau}X_i \leq (1-\varepsilon)(\tau-k+1)\overline{X} \right\}
    \leq \E\left( \E\left( e^{t(X-(1+\varepsilon)\overline{X})} \right)^{\tau-k} \right).
  \end{equation}
  Further, if 
  $\tau-k\geq l$ for some $l\in\N$, then there exists
  $0<\gamma<1$ not depending on $\tau,k,l$ such that
  \begin{equation}\label{eq:secondUpperEstimate}
    \mathbb{P}\left\{ \sum_{i=k}^{\tau}X_i \geq (1+\varepsilon)(\tau-k+1)\overline{X} \right\}
    \leq \gamma^{l}
  \end{equation}
  and
  \begin{equation}\nonumber
    \mathbb{P}\left\{ \sum_{i=k}^{\tau}X_i \leq (1-\varepsilon)(\tau-k+1)\overline{X} \right\}
    \leq \gamma^{l}.
  \end{equation}
\end{lma}
\begin{proof}
  In what follows we write $\{\mathcal{F}_s\}_{s \geq 0}$ for the natural filtration of our event
  space.  We prove \eqref{eq:firstUpperEstimate} and \eqref{eq:secondUpperEstimate}.  The remaining
  estimates are proved similarly and we omit the details.  We rearrange the left hand side of 
  \eqref{eq:firstUpperEstimate} and multiply by $t>0$ to obtain
  \begin{align*}
    \mathbb{P}\left\{ \sum_{i=k}^{\tau}X_i \geq (1+\varepsilon)(\tau-k+1)\overline{X} \right\} &= 
    \mathbb{P}\left\{\sum_{i=k}^{\tau}t(X_i-(1+\varepsilon)\overline{X}) \geq 0 \right\}\\
    &= \mathbb{P}\left\{ \exp\left[ \sum_{i=k}^{\tau}Y_i \right]\geq 1 \right\},
    \intertext{with $Y_i=t X_i- t(1+\varepsilon)\overline{X}$. Using Markov's inequality and continuing,}
    &\leq \E\left(\exp\left[\sum_{i=k}^{\tau}Y_i\right] \right)\\
    &= \E\left( \E\left( \exp \left[ \sum_{i=k}^\tau Y_i \right] \Big| \mathcal{F}_{\tau-1}\right)
    \right)\\
    & = \E\left( \E\left( \exp Y_\tau \mid \mathcal{F}_{\tau-1}\right)\E\left( \exp\left[ \sum_{i=k}^{\tau-1}Y_i \right]
    \Big| \mathcal{F}_{\tau-1} \right) \right)\\
    & = \E \left( \E\left( \exp Y_0 \right)\E \left( \E\left( \exp\left[ \sum_{i=k}^{\tau-1}Y_i
    \right]\Big| \mathcal{F}_{\tau-2} \right)\Big|\mathcal{F}_{\tau-1} \right) \right)\\
    & = \E \left( \E\left( \exp Y_0 \right)^2 \E \left( \E\left( \exp\left[ \sum_{i=k}^{\tau-2}Y_i
    \right]\Big| \mathcal{F}_{\tau-2} \right)\Big|\mathcal{F}_{\tau-1} \right) \right)\\
    &  \vdots\\
    & = \E \left( \E\left( \exp Y_0 \right)^{\tau-k}  \right)\\
    & = \E\left( \E\left( e^{t(X-(1+\varepsilon)\overline{X})} \right)^{\tau-k} \right)\,,
  \end{align*}
  as required.   

To prove \eqref{eq:secondUpperEstimate}  we consider 
  \begin{equation*}
    \gamma_t = \E\left( e^{t\left( X_0-(1+\varepsilon)\overline{X} \right)} \right).
  \end{equation*}
  Since $X$ is discrete, we can differentiate with respect to $t$ for all $t\in\R$, and get
  \begin{align*}
    \frac{d}{dt}\E\left( e^{t\left( X_0-(1+\varepsilon)\overline{X} \right)}
    \right)\Big\rvert_{t=0}
    &= \E\left( \frac{d}{dt}  e^{t\left( X_0-(1+\varepsilon)\overline{X} \right)}
    \right)\Big\rvert_{t=0}\\ \\
    &= 
    \E\left( \left( X_0-(1+\varepsilon)\overline{X} \right) e^{t\left( X_0-(1+\varepsilon)\overline{X} \right)}
    \right)\Big\rvert_{t=0}\\
    & = \E\left( X_0-(1+\varepsilon)\overline{X} \right)=-\varepsilon \overline{X}<0.
  \end{align*}
  Thus, since $\gamma_0=1$, there exists $t>0$ such that $0<\gamma_t<1$. Note that $t$ (and thus $\gamma_t$) does not depend on
  on $\tau,k,l$ and we can now use \eqref{eq:firstUpperEstimate} together with the assumption that $\tau-k \geq l$ to
  obtain \eqref{eq:secondUpperEstimate}, where $\gamma=\gamma_t$.
\end{proof}

Note that from the definitions
of $k_1^\omega(R)$ and $k_2^\omega(R)$ we can conclude that there exist constants
$c_1,c_\theta>1$ such that for sufficiently small $R$,
\[
  k_1^\omega(R) \geq c_1 k_2^\omega(R), \quad k_1^\omega(R^{1/\theta})\geq c_\theta k_1^\omega (R),
  \quad\text{ and }\quad k_2^\omega(R^{1/\theta}) \geq c_\theta k_2^\omega(R).
\]
The relationship between $k_1^\omega(R)$ and $k_2^\omega(R^{1/\theta})$ is more complicated and
depends heavily on $\omega$ and $R$. However, probabilistically we can say more.
Let $\varepsilon>0$ and $q\in\N$. 
Note that, taking logarithms,
\[
  \mathbb{P}\left\{ \prod_{i=1}^q n_{\omega_i}^{-1} \leq (\nbar)^{-(1+\varepsilon)q} \right\}
  =\mathbb{P} \left\{ \sum_{i=1}^q \log n_{\omega_i} \geq (1+\varepsilon)q\log \nbar \right\} 
\]
and therefore, by Lemma~\ref{lma:randomUpperLemma}, there exists $0<\gamma<1$
such that
\[
  \mathbb{P}\left\{ \prod_{i=1}^q n_{\omega_i}^{-1} \leq (\nbar)^{-(1+\varepsilon)q} \right\}\leq
  \gamma^{q-1}.
\]
Now, summing over $q$, we obtain
\[
  \sum_{q=1}^\infty \mathbb{P}\left\{ \prod_{i=1}^q n_{\omega_i}^{-1} \leq (\nbar)^{-(1+\varepsilon)q} \right\}\leq
  \sum_{q=1}^\infty \gamma^{q-1} < \infty.
\]
Thus, by the Borel-Cantelli Lemma, almost surely there are at most finitely many $q$ such these
events occur. We can similarly argue for a lower bound and conclude that for almost
all $\omega\in\Omega$ there exists
$q_\omega$ such that,
\begin{equation}\label{eq:nproduct}
  (\nbar)^{-(1+\varepsilon) q}\leq \prod_{i=1}^q n_{\omega_i}^{-1} \leq
  (\nbar)^{-(1-\varepsilon) q},
\end{equation}
for all $q\geq q_\omega$.
Analogously,
\begin{equation}\label{eq:mproduct}
  (\mbar)^{-(1+\varepsilon) q}\leq \prod_{i=1}^q m_{\omega_i}^{-1} \leq
  (\mbar)^{-(1-\varepsilon) q},
\end{equation}
almost surely for all $q$ large enough. Without loss of generality we can assume $q_\omega$ to be
identical for both products.
Since $k_2^\omega(R)\geq -c\log R$ for some $c>0$ not depending on $\omega,R$, we see that there
almost surely also exists an $R_\omega$ such that \eqref{eq:nproduct} and \eqref{eq:mproduct} hold
for all $q\geq k_2^\omega(R)$, where $0<R\leq R_\omega$.

Given these bounds we can determine the probabilistic relationship between
$k_1^\omega(R)$ and $k_2^\omega(R^{1/\theta})$.
Let $R_\omega$ be as above. Then by the definitions of $k_1^\omega(R)$ and
$k_2^\omega(R^{1/\theta})$ we get, for all $R\leq R_\omega$,
\[
  (\mbar)^{-(1+\varepsilon)k_1^\omega(R)}\leq\prod_{i=1}^{k_1^\omega(R)}
  m_{\omega_i}^{-1}\leq R < n_{\max}^\theta\left(
  \prod_{i=1}^{k_2^\omega(R^{1/\theta})} n_{\omega_i}^{-1} \right)^{\theta}
  \leq n_{\max}^\theta (\nbar)^{-\theta(1-\varepsilon) k_2^\omega(R^{1/\theta})}
\]
and, after rearranging,
\begin{equation}\label{eq:firstkRatio}
  \frac{k_1^\omega(R)}{k_2^\omega(R^{1/\theta})} > \theta \frac{1-\varepsilon}{1+\varepsilon}
  \frac{\log\nbar}{\log\mbar}-\theta\frac{\log
    n_{\max}}{(1+\varepsilon)\,k_2^\omega(R^{1/\theta})\log\mbar}.
  \end{equation}
Similarly, by considering the complementary inequalities
\[
\left(\prod_{i=1}^{k_2^\omega(R^{1/\theta})}
  n_{\omega_i}^{-1}\right)^\theta \leq R < m_{\max} 
  \prod_{i=1}^{k_1^\omega(R)} m_{\omega_i}^{-1} \,,
\]
we obtain
\begin{equation}\label{eq:secondkRatio}
  \frac{k_1^\omega(R)}{k_2^\omega(R^{1/\theta})}< \theta  \frac{1+\varepsilon}{1-\varepsilon}
  \frac{\log\nbar}{\log\mbar}-\frac{\log
    m_{\max}}{(1-\varepsilon)\,k_2^\omega(R^{1/\theta})\log\mbar}.
\end{equation}

Now $\varepsilon>0$ was arbitrary and the last term in \eqref{eq:firstkRatio} and
\eqref{eq:secondkRatio} vanishes as $R_\omega$ decreases. Therefore, for all $\delta>0$ and almost
all $\omega\in\Omega$, there exists sufficiently small $R_\omega>0$ such that
\begin{equation}\label{eq:limitRatio1}
  (1-\delta)\frac{\theta\log\nbar}{\log\mbar}
  \leq \frac{k_1^\omega(R)}{k_2^\omega(R^{1/\theta})}\leq
  (1+\delta)\frac{\theta\log\nbar}{\log\mbar}\,,
\end{equation}
for all $R<R_\omega$.
Moreover, using the much simpler relationships derived above, we can assume without loss of generality that $R_\omega$ is small enough such that 
\begin{equation}\label{eq:limitRatio2}
  (1-\delta)\frac{\log\nbar}{\log\mbar}\leq\frac{k_1^\omega(R)}{k_2^\omega(R)}\leq
  (1+\delta)\frac{\log\nbar}{\log\mbar},
\end{equation}
\begin{equation}\label{eq:limitRatio3}
  (1-\delta)\theta \leq \frac{k_1^\omega(R)}{k_1^\omega(R^{1/\theta})} \leq (1+\delta)\theta
  \quad\text{and}\quad
  (1-\delta)\theta \leq \frac{k_2^\omega(R)}{k_2^\omega(R^{1/\theta})} \leq (1+\delta)\theta
\end{equation}
all hold simultaneously for all $R<R_\omega$.

\subsubsection{The upper bound for $\theta< \log\mbar/\log\nbar$}
We assume throughout that $R_\omega$ is small enough for all inequalities in the last section to
hold simultaneously (almost surely).  Also, let $\delta>0$ be small enough such that the inequalities at the end of the previous section are all bounded away from $1$. That is we choose $\delta>0$ such that $(1+\delta)\theta<1$, $(1-\delta)\frac{\log\nbar}{\log\mbar}>1$ and, especially relevent to this section, \eqref{eq:limitRatio1} and $\theta< \log\mbar/\log\nbar$ imply that we can choose
$\delta>0$ sufficiently small such that  $k_1^\omega(R)<k_2^\omega(R^{1/\theta})$ almost surely for all $R<R_\omega$.

Let $\varepsilon>0$ and consider the geometric average given by
\begin{equation}\label{eq:firstAverage}
  \left(\Cbar^{k_1^\omega(R)-k_2^\omega(R)}\;\Nbar^{k_2^\omega(R^{1/\theta})-k_1^\omega(R)}
  \;\Bbar^{k_1^\omega(R^{1/\theta})-k_2^\omega(R^{1/\theta})}\right)^{1+\varepsilon}.
\end{equation}
We want to determine the probability that there exists an approximate $R$ square at a given level such that we need
more than the estimate in \eqref{eq:firstAverage} many $R^{1/\theta}$ squares to cover it.
Note that for \eqref{eq:upperCounting1} to be larger than \eqref{eq:firstAverage}, at least one of
the products must exceed the corresponding power of the average.  Therefore,
\begin{multline}
  \mathbb{P}\left\{ N\left( Q\cap F_\omega, R^{1/\theta} \right) \geq 4
  \left(\Cbar^{k_1^\omega(R)-k_2^\omega(R)}\;\Nbar^{k_2^\omega(R^{1/\theta})-k_1^\omega(R)}
  \;\Bbar^{k_1^\omega(R^{1/\theta})-k_2^\omega(R^{1/\theta})}\right)^{1+\varepsilon}
   \right\}
   \\
   \leq
   \mathbb{P}\left\{ \left(\prod_{l=k_2^\omega(R)+1}^{k^\omega_1(R)} C_{\omega_l} \right)\geq 
 \Cbar^{(1+\varepsilon)(k_1^\omega(R)-k_2^\omega(R))}\right\}
 +
  \mathbb{P}\left\{ \left(\prod_{l=k^\omega_1(R)+1}^{k^\omega_2\left(R^{1/\theta}\right)}  N_{\omega_l} \right) \geq 
 \Nbar^{(1+\varepsilon)(k_2^\omega(R^{1/\theta})-k_1^\omega(R))}\right\}
 \\+
\mathbb{P}\left\{\left( \prod_{l=k^\omega_2\left(R^{1/\theta}\right)+1}^{k_1^\omega\left(R^{1/\theta}\right)} 
B_{\omega_l} \right) \geq 
\Bbar^{(1+\varepsilon)(k_1^\omega(R^{1/\theta})-k_2^\omega(R^{1/\theta}))}\right\}.
\end{multline}
Let us start by analysing the event involving $C_{\omega_l}$. We want to show that the product can
only exceed the average behaviour at most finitely many times almost surely.  That is, given $q\in\N$, 
we want to estimate 
\begin{equation} \label{event111}
  \mathbb{P}\left\{    \prod_{l=k_2^\omega(R)+1}^{k^\omega_1(R)} C_{\omega_l} \geq 
  \Cbar^{(1+\varepsilon)(k_1^\omega(R)-k_2^\omega(R))}  \text{ for some $R \in (0,R_\omega)$ such that } k_2^\omega(R) = q \right\}.
\end{equation}
Notice that $k_1^\omega(R)$ is a stopping time and, by \eqref{eq:limitRatio2} and our assumption that $R_\omega$ and $\delta$ are chosen sufficiently small, $k_1^\omega(R) \geq
(1-\delta)\log\nbar/\log\mbar \ k_2^\omega(R)$ and $c:=(1-\delta)\log\nbar/\log\mbar -1 >0$.  
Using Lemma~\ref{lma:randomUpperLemma}, we can bound \eqref{event111} above by 
\begin{align*}
 \mathbb{P}\left\{  \bigcup_{\substack{q': \exists R \in (0,R_\omega) \\ k_2^\omega(R) = q \text{ and } k_1^\omega(R) = q'}}  \left\{ \sum_{l=q+1}^{q'} \log C_{\omega_l}  \geq
  (1+\varepsilon)(q'-q)\log\Cbar  \right\} \right\}
  \leq L\gamma^{c(q-1)}\,,
\end{align*}
for some $0<\gamma<1$ where $L>0$ is a deterministic constant corresponding to the number of
possible values for $k_1^\omega(R)$, given $k_2^\omega(R)$. Since
\[
  \sum_{q=1}^\infty  L \gamma^{c(q-1)} < \infty\,,
\]
the Borel-Cantelli lemma implies that the product can exceed the average behaviour only finitely many times almost surely.
The argument for $N_{\omega_l}$ and $B_{\omega_l}$ is identical due to the ratios given  
in \eqref{eq:limitRatio1}, \eqref{eq:limitRatio2}, and \eqref{eq:limitRatio3}. Therefore, there
almost surely
exists $q$ large enough---and hence $R_\omega'$ small enough---such that 
\[
  N\left( Q\cap F_\omega, R^{1/\theta} \right) \leq 4
  \left(\Cbar^{k_1^\omega(R)-k_2^\omega(R)}\;\Nbar^{k_2^\omega(R^{1/\theta})-k_1^\omega(R)}
  \;\Bbar^{k_1^\omega(R^{1/\theta})-k_2^\omega(R^{1/\theta})}\right)^{1+\varepsilon}\,,
\]
for all $0<R<R_\omega'$. Using \eqref{eq:limitRatio1}, \eqref{eq:limitRatio2}, and
\eqref{eq:limitRatio3} again, we obtain 
\[
  k_1^\omega(R)-k_2^\omega(R)\leq \left(
  (1+\delta)\frac{\log\nbar}{\log\mbar}-1\right)k_2^\omega(R)
\]
\[
  k_2^\omega(R^{1/\theta})-k_1^\omega(R) \leq \left(\frac{1+\delta}{1-\delta}\theta^{-1}  -
  (1+\delta)\frac{\log\nbar}{\log\mbar} \right)k_2^\omega(R)
\]
and
\[
  k_1^\omega(R^{1/\theta})-k_2^\omega(R^{1/\theta}) \leq \left(
  \frac{1+\delta}{1-\delta}\frac{\log\nbar}{\log\mbar} - (1-\delta)^{-1}
  \right)\theta^{-1}k_2^\omega(R).
\]
Now, using $k_2^\omega(R) \leq -(1+\delta)\log R / \log\nbar$, we rearrange,
\[
  \Cbar^{k_1^\omega(R)-k_2^\omega(R)}\leq
  \Cbar^{-(1+\delta)^2\log R / \log\mbar - (-(1+\delta)\log R /\log\nbar)} =
  R^{(1-1/\theta)s_c},
\]
where
\[
  s_c =
  (1+\delta)^2\frac{\log\Cbar}{\log\mbar}\frac{\theta}{1-\theta}-(1+\delta)\frac{\log\Cbar}{\log\nbar}\frac{\theta}{1-\theta}
  \quad
  \to
  \quad
  s_C := \frac{\theta}{1-\theta}\left(\frac{\log\Cbar}{\log\mbar} - 
  \frac{\log\Cbar}{\log\nbar}
  \right),\quad\text{as}\quad \delta\to0.
\]
We rearrange the other terms similarly  to obtain
\[
  N\left( Q\cap F_\omega, R^{1/\theta} \right) \leq 4 R^{(1-1/\theta)(1+\varepsilon)(s_c+s_n+s_b)},
\]
where 
\[
  s_n = -(1+\delta)^2 \frac{\log\Nbar}{\log\mbar}\frac{\theta}{1-\theta} +
  \frac{(1+\delta)^2}{1-\delta}\frac{\log\Nbar}{\log\nbar}\frac{1}{1-\theta}
  \quad
  \to
  \quad
  s_N:=\frac{1}{1-\theta}\left(\frac{\log\Nbar}{\log\nbar}-\theta
  \frac{\log\Nbar}{\log\mbar}
  \right),\quad\text{as}\quad \delta\to0,
\]
and
\[
  s_b =
  -\frac{1+\delta}{1-\delta}\frac{\log\Bbar}{\log\nbar}\frac{1}{1-\theta}+
  \frac{(1+\delta)^2}{1-\delta} \frac{\log\Bbar}{\log\mbar}\frac{1}{1-\theta}
  \quad
  \to
  \quad
  s_B:=\frac{1}{1-\theta}\left(
  \frac{\log\Bbar}{\log\mbar}-\frac{\log\Bbar}{\log\nbar}
  \right),\quad\text{as}\quad \delta\to0.
\]
For arbitrary  $\varepsilon'>0$ we may assume  $\delta>0$ is small enough such that $s_c+s_n+s_b \leq
(1+\varepsilon')(s_C+s_N+s_B)$. Note that
\begin{equation}\label{eq:dimensionCase1}
  s:=s_C+s_N+s_B = \frac{1}{1-\theta}\left[ 
    \left( \frac{\log\Bbar}{\log\mbar}+\frac{\log\Nbar/\Bbar}{\log\nbar}
    \right)
    - \theta\left(
    \frac{\log\Nbar/\Cbar}{\log\mbar}+
  \frac{\log\Cbar}{\log\nbar} \right) \right].
\end{equation}
We can therefore conclude that, almost surely, every approximate square of length $R<R_\omega$ can be covered by fewer than
\begin{equation}\nonumber
 4 R^{(1-1/\theta)(1+\varepsilon)(1+\varepsilon')s}
\end{equation}
sets of diameter $R^{1/\theta}$. Thus the Assouad spectrum is bounded above by
$(1+\varepsilon)(1+\varepsilon')s$ and by the arbitrariness of $\varepsilon,\varepsilon'$, also by
$s$.
\qed

\subsubsection{The upper bound for $\theta> \log\mbar/\log\nbar$}
The proof for this case follows along the same lines as $\theta < \log\mbar/\log\nbar$ and
we will only sketch their differences.
First note that $\theta> \log\mbar/\log\nbar$ implies the almost sure existence of a small enough
$R_\omega$ such that 
\[
  k_1^\omega(R)\geq (1-\delta)\theta \frac{\log\nbar}{\log\mbar} k_2^\omega(R^{1/\theta})\,,
\]
for all $R<R_\omega$. Again we assume without loss of generality that $R_\omega$ is chosen such that
\eqref{eq:limitRatio1}, \eqref{eq:limitRatio2}, and \eqref{eq:limitRatio3} are satisfied for a given
$\delta>0$.  We also choose $\delta>0$ small enough to ensure that
\[
(1-\delta)\theta \frac{\log\nbar}{\log\mbar} > 1.
\]
Let $\varepsilon>0$ and consider the geometric average
\[
  \left( \Cbar^{k_2^\omega(R^{1/
  \theta})-k_2^\omega(R)}\Bbar^{k_1^\omega(R^{1/\theta})-k_1^\omega(R)}
  \right)^{1+\varepsilon}.
\]
We compare the upper bound given in \eqref{eq:upperCounting2} with the average above and obtain
\begin{multline}
  \mathbb{P}\left\{ N\left( Q\cap F_\omega, R^{1/\theta} \right) \geq 4
  \left(\Cbar^{k_2^\omega(R^{1/\theta})-k_2^\omega(R)}  
  \;\Bbar^{k_1^\omega(R^{1/\theta})-k_1^\omega(R)}\right)^{1+\varepsilon}
   \right\}
   \\
   \leq
   \mathbb{P}\left\{ \left(\prod_{l=k_2^\omega(R)+1}^{k^\omega_2(R^{1/\theta})} C_{\omega_l} \right)\geq 
   \Cbar^{(1+\varepsilon)(k_2^\omega(R^{1/\theta})-k_2^\omega(R))}\right\}
 +
\mathbb{P}\left\{\left( \prod_{l=k^\omega_1\left(R\right)+1}^{k_1^\omega\left(R^{1/\theta}\right)} 
B_{\omega_l} \right) \geq 
\Bbar^{(1+\varepsilon)(k_1^\omega(R^{1/\theta})-k_1^\omega(R))}\right\}.
\nonumber
\end{multline}

Now using the same ideas as before, noting that $k_1^\omega(\cdot)$ and $k_2^\omega(\cdot)$ are stopping times, 
we can conclude that for almost every $\omega\in\Omega$ there exists $R_\omega$ such that 
\[
 N\left( Q\cap F_\omega, R^{1/\theta} \right) \leq 4 \left( \Cbar^{k_2^\omega(R^{1/
  \theta})-k_2^\omega(R)}\Bbar^{k_1^\omega(R^{1/\theta})-k_1^\omega(R)}
  \right)^{1+\varepsilon}\,,
\]
for all $R<R_\omega$. Using the estimates for $k_1^\omega(R^{1/\theta})/k_1^\omega(R)$ and
$k_2^\omega(R^{1/\theta})/k_2^\omega(R)$ in \eqref{eq:limitRatio3} we see that there exists
$\varepsilon'>0$ such that for sufficiently small $R$,
\[
  \left( \Cbar^{k_2^\omega(R^{1/
  \theta})-k_2^\omega(R)}\Bbar^{k_1^\omega(R^{1/\theta})-k_1^\omega(R)}
  \right)^{1+\varepsilon} \leq R^{(1-1/\theta)(1+\varepsilon)(1+\varepsilon')s},
\]
where
\[
  s=\frac{\log\Bbar}{\log\mbar}+\frac{\log\Cbar}{\log\nbar}.
\]
As before, this is sufficient to prove that for $\theta>\log\mbar/\log\nbar$, there
almost surely exists $R_\omega$ such that  all approximate $R$-squares with $R< R_\omega$ can be covered by fewer than 
\[
  4
  R^{(1-1/\theta)(1+\varepsilon)(1+\varepsilon')s}
\]
 sets of diameter $R^{1/\theta}$. This proves that $\dim_\textup{A}^\theta F_\omega \leq
(1+\varepsilon)(1+\varepsilon')s$ almost surely, and hence, by arbitrariness of
$\varepsilon,\varepsilon'>0$ that $\dim_\textup{A}^\theta F_\omega \leq
s$ almost surely, as required. \qed

\subsubsection{The lower bound for $\theta< \log\mbar/\log\nbar$}

To prove almost sure  lower bounds for $\dim_\textup{A}^\theta F_\omega$ we need to show that almost
surely there exists a sequence $R_i \to 0$ such that for each $i$ there is an approximate
$R_i$-square which requires at least a certain number of sets of diameter $R_i^{1/\theta}$ to cover
it.

 Let $\theta< \log\mbar/\log\nbar$ and, as before, we choose $\delta>0$  small enough such that 
$k_1^\omega(R)<k_2^\omega(R^{1/\theta})$ almost surely for all small enough $R$.   Let $\varepsilon>0$, and given  $q\in\N$ and $\omega \in \Omega$, let
\[
R_q = \prod_{l=1}^{q} n_{\omega_l}^{-1}\,,
\]
noting that $k_2^\omega(R_q)=q$ and $R_q \to 0$ as $q \to \infty$.  We have
\begin{multline}\label{eq:lowerProbBound}
  \mathbb{P}\Bigg\{ N\left( Q\cap F_\omega, R_q^{1/\theta} \right) \geq K
  \left(\Cbar^{k_1^\omega(R_q)-k_2^\omega(R_q)}\;\Nbar^{k_2^\omega(R_q^{1/\theta})-k_1^\omega(R_q)}
  \;\Bbar^{k_1^\omega(R_q^{1/\theta})-k_2^\omega(R_q^{1/\theta})}\right)^{1-\varepsilon}    \Bigg\}
   \\
   \geq 1-
   \mathbb{P}\left\{ \prod_{l=k_2^\omega(R_q)+1}^{k^\omega_1(R_q)} C_{\omega_l} \leq 
   \Cbar^{(1-\varepsilon)(k_1^\omega(R_q)-k_2^\omega(R_q))}\quad\text{ or }\quad
   \prod_{l=k^\omega_1(R_q)-1}^{k^\omega_2\left(R_q^{1/\theta}\right)}  N_{\omega_l}   \leq
 \Nbar^{(1-\varepsilon)(k_2^\omega(R_q^{1/\theta})-k_1^\omega(R_q))}\right.
 \\
 \left. \text{or}\quad \prod_{l=k^\omega_2\left(R_q^{1/\theta}\right)+1}^{k_1^\omega\left(R_q^{1/\theta}\right)} 
B_{\omega_l}   \leq
\Bbar^{(1-\varepsilon)(k_1^\omega(R_q^{1/\theta})-k_2^\omega(R_q^{1/\theta}))} \right\}.
\end{multline}
The last term is bounded above by
\begin{multline}\nonumber
  \mathbb{P}\left\{ \prod_{l=k_2^\omega(R_q)+1}^{k^\omega_1(R_q)} C_{\omega_l} \leq 
   \Cbar^{(1-\varepsilon)(k_1^\omega(R_q)-k_2^\omega(R_q))}\right\}\\
 +\mathbb{P}\left\{ \prod_{l=k^\omega_1(R_q)-1}^{k^\omega_2\left(R_q^{1/\theta}\right)}  N_{\omega_l}   \leq
 \Nbar^{(1-\varepsilon)(k_2^\omega(R_q^{1/\theta})-k_1^\omega(R_q))}\right\}\\
 +\mathbb{P}\left\{ \prod_{l=k^\omega_2\left(R_q^{1/\theta}\right)+1}^{k_1^\omega\left(R_q^{1/\theta}\right)} 
B_{\omega_l}  \leq
\Bbar^{(1-\varepsilon)(k_1^\omega(R_q^{1/\theta})-k_2^\omega(R_q^{1/\theta}))}\right\}
\end{multline}
and by Lemma~\ref{lma:randomUpperLemma} and the union estimate used above each probability is
bounded above by $L' \gamma^{c'q}$ for some constants $L', c'>0$ and $\gamma \in (0,1)$.  Thus,
there exists $q_0$ such that each term in the sum is bounded by $1/6$ for $q\geq q_0$ and thus
the probability on the left hand side of \eqref{eq:lowerProbBound} is bounded below by $1/2$ for $q\geq q_0$.

Denote the event on the left hand side of \eqref{eq:lowerProbBound} by $E_q$. Observe that the event
only depends on the values of $\omega_i$ for $i$ satisfying $q= k_2^\omega(R_q) \leq i \leq  k_1^\omega(R_q^{1/\theta})$ as
the latter bound is a stopping time. By virtue of construction, there exists an integer $d \geq 1$ such that $k_1^\omega\left(R_{d^iq}^{1/\theta}\right)< d^{i+1} q$ for all $q$. Therefore the events $\{E_q, E_{dq}, E_{d^2 q},\dots\}$ are pairwise independent. Further, by the above argument,
\[
  \sum_{i=0}^{\infty}\mathbb{P}(E_{d^i q_0})\geq \sum_{i=0}^{\infty} 1/2 =\infty
\]
and so by the Borel-Cantelli Lemmas $E_q$ happens infinitely often.  Therefore, adapting the
argument involving $s_c, s_b$ and $s_b$ from above,  we have proved that for all $\varepsilon'>0$
there almost surely exist infinitely many $q \in \mathbb{N}$ such that there exists an approximate
$R_q$-square $Q$ such that
\[
 N\left( Q\cap F_\omega, R_q^{1/\theta} \right) \geq K R_q^{(1-1/\theta)(1-\varepsilon') s},
\]
where $s=s_C+s_B+s_N$ is the target lower bound for the spectrum.  This completes the proof.  \qed

\section*{Acknowledgements}
This work was started while both authors were  resident at the  Institut Mittag-Leffler during the
2017  semester programme \emph{Fractal Geometry and Dynamics}.  They are grateful for the
stimulating environment.  Much of the work was subsequently carried out whilst JMF visited the
University of Waterloo in March 2018. He is grateful for the financial support, hospitality, and
inspiring research atmosphere.    \\

\end{document}